\documentclass[a4paper,10pt,reqno]{amsart}
\usepackage{latexsym}
\usepackage{ae}
\usepackage[english]{babel}
\usepackage[latin1]{inputenc}
\usepackage[OT1]{fontenc} 
\usepackage{graphicx}

\usepackage{epic, eepic}
				
\theoremstyle{plain}
\newtheorem{theorem}{Theorem}
\newtheorem*{theorem*}{Theorem}
\newtheorem{corollary}[theorem]{Corollary}
\newtheorem*{corollary*}{Corollary}

\newtheorem*{lemma*}{Lemma}

\newtheorem*{proposition*}{Proposition}

\newtheorem*{conjecture*}{Conjecture}

\theoremstyle{definition}
\newtheorem{definition}[theorem]{Definition}
\newtheorem*{definition*}{Definition}

\newtheorem*{example*}{Example}

\newtheorem*{problem*}{Problem}

\theoremstyle{remark}

\newtheorem*{remark*}{Remark}

\thanks{The work presented here was supported by grant
  no.\ 060005013 from the Icelandic Research Fund}

\title[Dyck paths, SYT, and Pattern Avoiding Permutations]{Dyck paths, Standard Young Tableaux, and Pattern Avoiding Permutations}

\author[H.\ H.\ Gudmundsson]{Hilmar Haukur Gudmundsson}

\def\dd{\makebox[1.1ex]{\rule[.58ex]{.71ex}{.15ex}}}

\begin{document}
\begin{abstract}

We present a generating function and a closed counting formula in two variables that enumerate a family of classes of permutations that avoid or contain an increasing pattern of length three and have a prescribed number of occurrences of another pattern of length three. This gives a refinement of some previously studied statistics, most notably one by Noonan.  The formula is also shown to enumerate a family of classes of Dyck paths and Standard Young Tableaux, and a bijection is given between the corresponding classes of these two families of objects.  Finally, the results obtained are used to solve an optimization problem for a certain card game.

\end{abstract}

\maketitle
\thispagestyle{empty}

\section{Introduction}\label{section-intro}
In this paper we look at three families of sets of combinatorial objects. The first family consists of all Dyck paths that start with at least $k$ upsteps, end with at least $p$ downsteps, and are concatenations of two or more components. So, except for a few trivial cases, that means such a path returns to the $x$-axis somewhere between the endpoints.

The second family consists of all Standard Young Tableaux with two rows, the top row being $d$ cells longer than the other. When $d=k+p$, the corresponding classes in the two families are equinumerous, which we demonstrate by giving a bijection. Consequently, we see that although we have mentioned two parameters for the paths, $k$ and $p$ (in addition to the one that marks size, $n$), only one is needed with regard to enumeration. So for example, the number of Dyck paths that start with at least $k+1$ upsteps and end with at least $p-1$ downsteps is the same as the number of Dyck paths that start with at least $k$ upsteps and ends with at least $p$ downsteps. We then derive a generating function and a closed counting formula for each class. The generating function for the classes of these families can then be given as $x^{k+p}C(x)^{k+p+1}$, where $C(x)$ stands for the Catalan generating function. 

The third family consists of several classes of permutations that either avoid an increasing subsequence of length three, or contain exactly one such subsequence, with restrictions on whether the elements of the subsequence are adjacent or not. A subsequence in this context will also be referred to as a \textit{pattern}. We then show for each of these classes that they satisfy the same recursion as a corresponding class in the family of Standard Young Tableaux, and we obtain a generating function and a closed counting formula for each one. When we set $d$ equal to one or zero, the generating function we obtain is the Catalan function itself, which counts the number of permutations that avoid a classical pattern of length three. The class that corresponds to $d=2$, contains the permutations that avoid an increasing subsequence of three elements and end with a descent. This result has already been established in \cite{Kitaev} but we provide a different proof. The classes corresponding to $d=4,5,6,7$ contain permutations which have a single occurrence of an increasing subsequence of three elements, with restrictions on whether these elements are allowed to be adjacent or not. The class corresponding to $d=5$ has already been dealt with in \cite{Noonan} and \cite{Little}, but we provide a different proof. In the final remarks, we explore how we can use the counting formula to find an optimal way of playing a card game.

\section{Background}
A \emph{Dyck path of semilength $n$} is a lattice path, that starts at the origin, ends in $(2n,0)$, and only contains upsteps $(1,1)$, denoted by $u$, and downsteps $(1,-1)$, denoted by $d$, in such a way that it never goes below the $x$-axis. We say a Dyck path \textit{returns} to the $x$-axis wherever it touches it.

\begin{definition}
Let $\mathcal{P}_{k,p}$ be the class of all Dyck paths that start with at least $k$ upsteps and end with at least $p$ downsteps. We say that a path $P$ in $\mathcal{P}_{k,p}$ is \textit{reducible in $\mathcal{P}_{k,p}$} if it is a concatenation of two or more Dyck paths such that at least one of the components in the concatenation is of semilength at least $k$, and another of semilength  at least $p$. If it is clear which class $\mathcal{P}_{k,p}$ we are referring to, we simply refer to $P$ as \textit{reducible}. We denote the class that consists of all the reducible paths in $\mathcal{P}_{k,p}$ by $\mathcal{D}_{k,p}$
\end{definition}

An example of a Dyck path with two components is given in Figure \ref{fig1}. The components meet in the vertex the arrow points to.
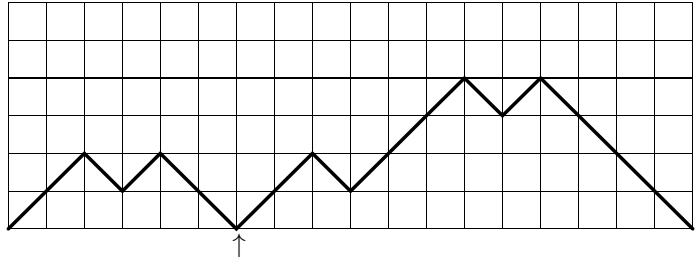
\begin{figure}
\newcommand\p{\circle*{1}}

\setlength{\unitlength}{1mm}
\begin{picture}(140,30)(0,0)
  \allinethickness{.03mm}
  \put(20,0){
    \put(0,0){\grid(90,30)(5,5)}
    \allinethickness{.5mm}
    \path(0,0)(5,5)(10,10)(15,5)(20,10)(25,5)(30,0)(35,5)(40,10)(45,5)
    (50,10)(55,15)(60,20)(65,15)(70,20)(75,15)(80,10)(85,5)(90,0)
    \put(29.5,-3){$\uparrow$}

}
\end{picture}
\caption{A Dyck path with two components.}
\label{fig1}
\end{figure}

The family of classes that was mentioned in Section 1 consists of the classes $\mathcal{D}_{k,p}$, for all nonnegative $p$ and $k$. When both $k$ and $p$ are greater than zero, every path belonging to $\mathcal{D}_{k,p}$ must touch the $x$-axis somewhere between the endpoints, since it must be a concatenation of at least two nonzero components, one starting with at least $k$ consecutive upsteps, and one ending with at least $p$ consecutive downsteps. When both $k$ and $p$ are equal to zero, it means that the empty path is included. This also means that one or more components can be empty, so the paths are not required to touch the $x$-axis, as was the case earlier, and so we see that this is simply the class of ordinary Dyck paths. When $k$ is equal to zero and $p$ is equal to one, or vice versa, we have the class of all non-empty Dyck paths, since now there must be at least one upstep (or downstep), belonging to the path.

A \emph{Standard Young Tableau} of shape $(2,n)$ consists of two horizontal rows of cells, one on top of the other, such that the top row is as long, or longer than the bottom row. Each cell is then assigned a unique number from $\left\{1,2,\ldots,2n\right\}$ so that the numbers in each row are strictly increasing from left to right, and the numbers in each column from top to bottom are also strictly increasing. An example of a Standard Young Tableau of semisize 9 is given in Figure \ref{fig5}. See chapter 6 in \cite{Bona2} for further explanations and diagrams.

\begin{definition}
Let $S$ be a rectangular Standard Young Tableau with two rows of semisize $n$. We say $S$ is \textit{reducible} if there exists an $i\leq n$ such that every number $1,2,\ldots,2i-2$ is included in a cell to the left of the $ith$ column. If $S$ does not satisfy this, then we say $S$ is irreducible.
\end{definition}

The Standard Young Tableau in Figure \ref{fig5} is reducible. The cells to the left of the arrow contain all the values in $\left\{1,2,\ldots,6\right\}$.

\begin{definition}
We define $\mathcal{Y}_{d}$ as the class of all Standard Young Tableaux that consist of two rows, one being $d$ cells longer than the other.
\end{definition}

\begin{figure}
\newcommand\p{\circle*{1}}

\setlength{\unitlength}{1mm}
\begin{picture}(140,30)(33,-10)

\allinethickness{.03mm}

\put(0,-152){
\put(12.5,-2){
\put(60,150){\grid(45,5)(5,5)}
\put(60,145){\grid(45,5)(5,5)}
\put(63,152){\makebox(0,0){1}}
\put(68,152){\makebox(0,0){2}}
\put(73,152){\makebox(0,0){5}}
\put(78,152){\makebox(0,0){7}}
\put(83,152){\makebox(0,0){8}}
\put(88,152){\makebox(0,0){9}}
\put(93,152){\makebox(0,0){10}}
\put(98,152){\makebox(0,0){11}}
\put(103,152){\makebox(0,0){12}}

\put(63,147){\makebox(0,0){3}}
\put(68,147){\makebox(0,0){4}}
\put(73,147){\makebox(0,0){6}}
\put(78,147){\makebox(0,0){13}}
\put(83,147){\makebox(0,0){14}}
\put(88,147){\makebox(0,0){15}}
\put(93,147){\makebox(0,0){16}}
\put(98,147){\makebox(0,0){17}}
\put(103,147){\makebox(0,0){18}}

\put(74,142){$\uparrow$}
}
}
\end{picture}
\caption{A reducible Standard Young Tableau.}
\label{fig5}
\end{figure}

Lastly, we need to be familiar with the Catalan generating function, which we will refer to as $C(x)$. It is given by $$C(x)= \frac{1-\sqrt{1-4x}}{2x},$$ which satisfies $$C(x)=\sum_{n\geq0}{\frac{1}{n+1}{2n\choose n}}x^{n}.$$ The number $\frac{1}{n+1}{2n\choose n}$ is called the \textit{nth Catalan number}.

\section{The main counting formula and a bijection}

With the definitions from Section 2 in mind we arrive at the following result.

\begin{theorem}
  The number of Dyck paths that start with at least $k$ upsteps, end
  with at least $p$ downsteps, and touch the x-axis somewhere
  between the first upstep and the last downstep, is given by the
  generating function
	\[D_{k,p}(x)=x^{p+k} C(x)^{k+p+1},
\]
where $C(x)$ is the Catalan function. The $nth$ coefficient of this generating function is given by the formula 
	\[
	D_n = \frac{k+p+1}{n+1}{2n-k-p \choose n}.
\]

\end{theorem}

\begin{proof}
  For given $p$ and $k$, the paths in question form a subset in the
  set $\mathcal{D}$ of all Dyck paths. If $k=p=0$, then $\mathcal{D}_{k,p}$ =
  $\mathcal{D}$. If $k=1$ and $p=0$, or $k=0$ and $p=1$ then
  $\mathcal{D}_{k,p}=\mathcal{D} \backslash \{\epsilon\}$, where
  $\epsilon$ stands for the empty path. Now let $k,p \geq 1$. Every path $P$ in $D_{k,p}$ thus has to touch the $x$-axis somewhere between the two
  endpoints. It can therefore be viewed as a concatenation of two or more Dyck paths. Moving from left to right, consider the first irreducible component. The path $P$ starts with at least $k$ upsteps, and returns to the $x$-axis at some point $(v_1,0)$. We also know that the last irreducible component ends with at least $p$-downsteps. If this last component starts at $(v_2,0)$, then the part of the path that starts in $(v_1,0)$ and ends in $(v_2,0)$ is an ordinary Dyck path with no restrictions (it might even be empty), see Figure
  1. We will count the number of elements in $\mathcal{D}_{k,p}$ by considering each of these three components of $P$ seperately. To begin with, we know that the number of ways we can construct the part of the path that begins at $(v_1,0)$ and ends at $(v_2,)$ is given by the Catalan function. Let us consider the set of Dyck paths that start with at least $k$ consecutive
  upsteps, and denote it by $\mathcal{D}_k$ (so $\mathcal{D}_1$ is the class of all non-empty
  Dyck paths). Since we can consider $P$ to be a concatenation of three paths, we can calculate the number of ways we can construct $P$ by considering the Cartesian product of the respective sets of paths, that is, $D_{k}, D$ and $D_{p}$. We thus obtain $\mathcal{D}_k = u \times \mathcal{D}_{k-1}\times d
  \times \mathcal{D}$, where $u$ stands for an upstep, and $d$ stands
  for a downstep. By iteration, we obtain
	\[
\mathcal{D}_k =   u^k \times \left[d\times \mathcal{D} \right]^{k}.
\]
If we translate this to generating functions we have
$D_{k}(x)=x^{k}D(x)^{k}$, where $D_{k}(x)$ is the generating function
for $\mathcal{D}_k$, and $D(x)$ is the Catalan generating function $C(x)$. The part of the original lattice path that ended
with $p$ consecutive downsteps obviously falls under the same
reasoning. Thus, $D_{k,p}(x) =
x^{k+p}D_{k}(x)D(x)D_{p}(x)= x^{k+p}D(x)^{k+p+1}$. The $nth$
coefficient, $D_n$ can then be extracted from the expression by
standard methods.\end{proof}

\enlargethispage{10mm}
By inspection, we realize that although we have at this point accumulated three different variables, $k,p$ and $n$, only two are needed. Should we denote $k+p$ by a new variable, say $d$, our results are simplified.
\begin{corollary}
Let $d=p+k$. Then the generating function and counting formula in Theorem 4 are equal to
$x^{d}C(x)^{d+1}$ and \begin{equation}
\frac{d+1}{n+1}{2n-d \choose n},\end{equation} respectively.
\end{corollary}

We will now see that the number of Standard Young Tableaux of shape $(n,n-d)$ is given by the same generating function as our lattice paths. 

\begin{theorem}
Let $d=k+p$. The class of Dyck paths that begin with at least $k$ successive upsteps, end with at least $p$ successive downsteps, and touch the $x$-axis at least once somewhere between the endpoints is equinumerous with the class of Standard Young Tableaux of shape $(n,n-d)$.
\end{theorem}

\begin{proof}
To begin with, we shall describe a well known bijection between Dyck paths and Standard Young Tableaux with two rows of equal length. By looking at a Dyck path of semilength $n$, we can enumerate the steps in a straightforward way, by assigning $1$ to the initial step, $2$ to the second step, and so on. Now, we fill out the corresponding Standard Young Tableau, by placing $1$ in the left upper corner, and then placing the next value immediately to the right of $1$ in the upper row if step number $2$ is an upstep, but immediately below $1$ if it is a downstep. And in this fashion we add the numbers to the tableau to the leftmost vacant cell in the upper row if the corresponding step in the Dyck path is an upstep, and in the leftmost vacant cell in the lower row if it is a downstep. We see that step number $2n$ then always ends up in the right lower corner of the tableau. Note that this means that if we have $k$ consecutive upsteps in the beginning of the path, the first $k$ numbers are all in the upper row in the tableau, and if it ends with $p$ downsteps the last $p$ numbers are all in the bottom row.

Recall that we have set $d=k+p$. When we consider the class of Standard Young Tableaux of shape $(n,n-d)$ we see that if we take a tableau from the class, and add $d$ cells to the bottom row, containing the values $2n-d+1,...,2n$, we have come across a trivial bijection between this class and the class of Standard Young Tableaux of shape $(n,n)$ where the values $2n-d+1,...,2n$ are fixed in the bottom row. With the preceding paragraph in mind, we know that we have a bijection from this latter class to the class of Dyck paths that end with at least $k+p$ downsteps. In order to prove our theorem, we construct a bijection $\zeta: \mathcal{D'}_{k+p} \rightarrow \mathcal{D}_{k,p}$, where $\mathcal{D'}_{k+p}$ is the class of Dyck paths that end with at least $k+p$ downsteps and $\mathcal{D}_{k,p}$ was defined in Definition 1. We then have a composite bijective function taking us from $\mathcal{Y}_{k,p}$ to $\mathcal{D}_{k,p}$.

Let $k,p$ be nonnegative, and choose a path from $\mathcal{D'}_{k,p}$. Locate the vertex on the path where only $p$ downsteps remain to be traversed, and name it $v_1$. Now, extend a horizontal line segment from $v_1$, to the left until it touches the lattice path again. Let us denote the vertex that the segment stops at by $v_2$. Now we remove the part of the lattice path that starts with $v_2$ and ends with $v_1$. That leaves the part of the lattice path that starts at the origin and ends at $v_2$, and $p$ successive downsteps that start at $v_1$. We then glue these two parts together, so that $v_1$ and $v_2$ become the same vertex, meaning that the part of the path that leads up to $v_2$ now has a segment of $p$ successive downsteps in the end which makes it a Dyck path. The final step is taking the segment we removed, inverting it from left to right, and gluing it to the left of the path we formed earlier, so that the final vertex of the inverted segment now becomes the same as the initial vertex of the path we just formed from the two separate parts. What remains is a Dyck path that starts with at least $k$ upsteps and ends with at least $p$ downsteps, which touches the $x$-axis at least once (where the latter gluing took part), and so we have formed a lattice path belonging to the class $\mathcal{D}_{k,p}$. The process is illustrated in Figure \ref{fig2}.

The inverse procedure works as follows. We have before us a Dyck path which begins with $k$ successive upsteps and ends with $p$ successive downsteps, and which is required to touch the $x$-axis somewhere between its initial and final vertices. Moving from left to right, we trace the lattice path until we find the first instance of it touching the $x$-axis. Let $s$ denote the vertex where this happens. Let $P_1$ denote the part of the lattice path that begins at the initial vertex and ends at $s$. Let $P_2$ be the part of the lattice path that starts in $s$ and ends at the final vertex. We now have two smaller lattice paths, $P_1$ and $P_2$ which we will rearrange to create a new lattice path. We know that $P_2$ has at least $p$ successive downsteps in the end, and now we erase the last $p$ (down)steps, to obtain $P_2'$ and attach the initial vertex of $P_1$ to the final vertex of $P_2'$, that is, we move $P_1$ up and to the right of $P_2'$ and glue these two parts together. We now have a path that ends precisely $p$ steps above the $x$-axis and we add $p$ downsteps, to the final vertex of $P_1$ and so it becomes a Dyck path, and thus representable by a standard Young tableau where the values $2n-d+1,...,2n$ are in the bottom row.\end{proof}

\begin{figure}
\newcommand\p{\circle*{1}}

\setlength{\unitlength}{1mm}
\begin{picture}(200,120)(33,-10)

\allinethickness{.03mm}

\put(12.5,-53){
\put(60,150){\grid(45,5)(5,5)}
\put(60,145){\grid(25,5)(5,5)}
\put(63,152){\makebox(0,0){1}}
\put(68,152){\makebox(0,0){3}}
\put(73,152){\makebox(0,0){4}}
\put(78,152){\makebox(0,0){5}}
\put(83,152){\makebox(0,0){9}}
\put(88,152){\makebox(0,0){10}}
\put(93,152){\makebox(0,0){11}}
\put(98,152){\makebox(0,0){12}}
\put(103,152){\makebox(0,0){13}}

\put(63,147){\makebox(0,0){2}}
\put(68,147){\makebox(0,0){6}}
\put(73,147){\makebox(0,0){7}}
\put(78,147){\makebox(0,0){8}}
\put(83,147){\makebox(0,0){14}}

}

\put(94.5,65){$\updownarrow$}

\put(12.5,-44){
\put(60,120){\grid(45,5)(5,5)}
\put(60,115){\grid(45,5)(5,5)}
\put(63,122){\makebox(0,0){1}}
\put(68,122){\makebox(0,0){3}}
\put(73,122){\makebox(0,0){4}}
\put(78,122){\makebox(0,0){5}}
\put(83,122){\makebox(0,0){9}}
\put(88,122){\makebox(0,0){10}}
\put(93,122){\makebox(0,0){11}}
\put(98,122){\makebox(0,0){12}}
\put(103,122){\makebox(0,0){13}}
\put(108,122){\makebox(0,0){}}
\put(113,122){\makebox(0,0){}}

\put(63,117){\makebox(0,0){2}}
\put(68,117){\makebox(0,0){6}}
\put(73,117){\makebox(0,0){7}}
\put(78,117){\makebox(0,0){8}}
\put(83,117){\makebox(0,0){14}}
\put(88,117){\makebox(0,0){15}}
\put(93,117){\makebox(0,0){16}}
\put(98,117){\makebox(0,0){17}}
\put(103,117){\makebox(0,0){18}}

}

\put(50,36){
\put(0,0){\grid(90,25)(5,5)}
\path(0,0)(5,5)(10,0)(15,5)(20,10)(25,15)(30,10)(35,5)(40,0)(45,5)(50,10)(55,15)(60,20)(65,25)(70,20)(75,15)(80,10)(85,5)(90,0)
}
\put(-55,0){
\put(105,0){\grid(90,25)(5,5)}
\path(105,0)(110,5)(115,10)(120,15)(125,10)(130,5)(135,0)(140,5)(145,0)(150,5)(155,10)(160,15)(165,10)(170,5)(175,0)(180,5)(185,10)(190,5)(195,0)
}

\put(94.5,30){$\updownarrow$}
\put(94.5,85){$\updownarrow$}

\allinethickness{.5mm}

\path(50,0)(55,5)(60,10)(65,15)(70,10)(75,5)(80,0)
\put(0,-24){
\path(100,70)(105,75)(110,80)(115,85)(120,80)(125,75)(130,70)
}
\end{picture}
\caption{Bijection between $\mathcal{Y}_{2,2}$ and $\mathcal{D}_{2,2}$.}
\label{fig2}
\end{figure}

\section{Patterns}

Let $\pi = a_1a_2,\ldots,a_n$, where $a_i = \pi(i) = \pi_i$, be a permutation
of length $n$. For a given $k \leq n$, let $p = b_1b_2,\ldots,b_k$ be
another permutation. We say that $\pi$ contains $p$ if there is a
sequence $\pi'$ within $\pi$ such that the relationship between the
elements in $\pi'$ is the same with regard to the usual order relation
on the integers as that of $p$. We refer to such a sequence as a
\textit{pattern}. For example, let $p=123$ and $\pi = 34215$. Then
$\pi$ has one occurrence of $p$, namely the sequence $345$. Another
pattern we might consider is $q=132$, but we see that $\pi$ has no
occurrence of $q$. We might in addition require that some of the
elements in the pattern are adjacent in the permutation. From now on,
we denote a pattern of three elements, where the elements are not required to be adjacent by $(a\dd b\dd c)$, and the set of all permutations of
length $n$ that avoid the pattern $(a\dd b\dd c)$ will be denoted by
$S_{n}(a\dd b\dd c)$. If two elements in the pattern are required to be
adjacent in the permutation, we denote it by removing the dash between
the two elements, so that if, for instance, $b,c$ are required to be
adjacent in the permutation, we denote the pattern as $(a\dd bc)$. For example, the permutation $\pi'=14253$ has a single occurrence of the pattern $(1\dd 2\dd 3)$, consisting of the numbers $1, 2$ and $3$. There is, however, no occurrence of the pattern $(1\dd 23)$ in $\pi'$. The
set of all permutations of length $n$ that avoid $(a\dd bc)$ is thus
denoted by $S_{n}(a\dd bc)$. We denote the class of all permutations that
contain the pattern $(a\dd bc)$ exactly $k$ times by $S^{k}_{n}(a\dd bc)$. If a
permutation $\pi$ belongs to this class then we write
$(a\dd bc)\pi=k$.

It is well known that the number of permutations that avoid the pattern $(1\dd 2\dd 3)$ is given by the Catalan numbers. Looking at Equation (1) in Corollary 5, we see that when $d=0$, we obtain the generating function and the counting formula for the Catalan numbers, and when we set $d=1$, we again obtain the Catalan sequence, without the extra $1$ denoting the number of Dyck paths of size zero. The following results will be established below. The number of permutations $\pi \in S_n(1\dd 2\dd 3)$ that end in a descent is given by Equation (1) in Corollary 5 with $d=2$. The number of permutations $\pi \in S_{n-1}^1 (1\dd 23)$ for which $(1\dd 2\dd 3)\pi = 1$ is given by Equation (1) with $d=4$. The number of permutations $\pi \in S_{n-2}^1 (1\dd 2\dd 3)$ is given by Equation (1) with $d=5$. The number of permutations $\pi \in S_{n-2}^1 (1\dd 2\dd 3)$ for which $(1\dd 23)\pi = 0$ is given by Equation (1) with $d=6$. And lastly, the number of permutations $\pi \in S_{n-2}^1 (1\dd 2\dd 3)$ for which $(1\dd 23)\pi = 0$ and $(12\dd 3)\pi=0$ is given by Equation (1) with $d=7$.

\subsection{Permutations that avoid an increasing subsequence of length three and end with a descent}

The following result was first established in \cite{Kitaev}, and corresponds to $d=2$ in Corollary 5.
\begin{theorem}
Let $a_n$ be the number of permutations $\pi \in S_n(1\dd 2\dd 3)$ which end in a descent (i.e. $\pi_{n-1}>\pi_n$).
Then
$$ \sum_{n\geq2} a_n x^n \;=\; x^2 C(x)^3 \;=\; \sum_{n\geq 2} \frac{3}{n+1} {2n-2 \choose n} x^n.$$
\end{theorem}

\begin{proof}
We know that the number of permutations of length $n$ without the latter restriction is given by the $nth$ Catalan number $C_n$. If a permutation $\pi$ avoids the pattern $(1\dd 2\dd 3)$, we see that if the smallest element is not in the rightmost or second rightmost position in $\pi$, then $\pi$ must end with a descent, or else we would have an occurrence of a $(1\dd 2\dd 3)$ pattern consisting of $1$ and the last two elements in the permutation. We also note that if $1$ is in the last position in $\pi$, then it must necessarily end with a descent. Therefore, the only way $\pi$ can avoid the pattern $(1\dd 2\dd 3)$ and not end with a descent is if $1$ is in the second last position. If we keep the smallest element fixed next to the last place in the permutation we have $n-1$ elements left to rearrange such that they avoid a $(1\dd 2\dd 3)$ pattern, which we can do in $C_{n-1}$ ways. Therefore, the number of permutations that avoid $(1\dd 2\dd 3)$ and end with a descent is given by $C_n-C_{n-1}$. 

Recall that we denote the generating function for the Catalan numbers by $C(x)$. One of the properties of this function is that it satisfies the identity $C(x)=1+xC(x)^2$. When we translate $C_n-C_{n-1}$ into generating functions, we can take advantage of this fact to obtain

\begin{eqnarray*}
 &x^{-1}(C(x)-xC(x)-1) \\
=&x^{-1}(C(x)-1)-C(x) \\
=&C(x)^2-C(x) \\
=&C(x)(C(x)-1) \\
=&xC(x)^3. \\
\end{eqnarray*}

We can then use standard methods to extract the $nth$ coefficient of the expansion of this function. \end{proof}

\subsection{Permutations with a single increasing subsequence of length three, where two elements are adjacent}

The generating function and counting formula in the following result correspond to $d=4$ in Corollary 5, where we divide the generating function by $x$, and replace $n$ with $n+1$ to provide a shift of $1$.

\begin{theorem}

Let $b_n$ be the number of permutations $\pi \in S_n^1 (1\dd 23)$ for which $(1\dd 2\dd 3)\pi \linebreak =1$. Then
$$\sum_{n\geq3} b_n x^n \;=\; x^3 C(x)^5 \;=\; \sum_{n\geq 3} \frac{5}{n+2} {2n-2 \choose n+1} x^n.$$\end{theorem}

\begin{proof}
We will exhibit a bijection between the class of permutations we want to count and another class of permutations that is more easily enumerated.

Let $\pi \in S_n^1(1\dd 23)$ with $(1\dd 2\dd 3)\pi = 1$. Let the single occurrence of the pattern in the permutation consist of $\pi_j=a, \pi_k=b$ and $\pi_{k+1}=c$, with $j\leq k$, we see that since there is no other occurrence of $(1\dd 2\dd 3)$ in $\pi$, all the elements to the right of $c$ must be smaller than $b$, or we would have more than one occurrence of a $(1\dd 2\dd 3)$ pattern. Likewise, we see that there can only be one element to the left of $b$ that is smaller than $b$, namely $a$ or we would once again have more than one occurrence of a $(1\dd 2\dd 3)$ pattern. Our bijection consists of replacing $a$ with $b$, and placing $a$ immediately to the right of $c$. We now have a permutation decomposable into a left and right component, which we shall refer to hereafter as $L$ and $R$ respectively, with the left component containing the values $b,b+1,\ldots,n$, and the right component containing the values $1,2,\ldots,b-1$. Furthermore, we see that $b$ cannot be the rightmost element in $L$ because $c$ has not been moved, and $b$ can only have been moved further to the left. In symbols, we would describe this as the permutation

$$\ell_1,\ell_2,\ldots,a,\ldots,b,c,r_1,r_2,\ldots,r_{b-2},$$ becoming the permutation $$\ell_1,\ell_2,\ldots,b,\ldots,c,a,r_1,r_2,\ldots,r_{b-2},$$ where $\ell_k\in L$, $r_k\in R$.

We can see that this is a bijection by describing the inverse function. Let $\pi$ be a permutation of length $n$ that is decomposable into two components, $L$ and $R$, such that the elements $1,2,\ldots,b-1$ are in $R$ and the elements $b,b+1,\ldots,n$ are in $L$. Furthermore, let us require that both $L$ and $R$ avoid the $(1\dd 2\dd 3)$ pattern, and that $b$ is not the rightmost element of $L$. Let us refer to the set of these permutations as $J_{n,b}$. The inverse of the function we described earlier on $J_{n,b}$ would then be to replace $b$ in $L$ with the rightmost element $r_\ell$ in $R$, and placing $b$ immediately to the left of the rightmost element $\ell_r$ in $L$.

We will now count the elements in $J_{n,i}$. Let $\pi \in J_{n,i}$, so the elements $1,2,\ldots,i-1$ are in $R$, and the elements $i,i+1,\ldots,n$ are in $L$. We can denote $\pi$ as the concatenation $LR$. We know that the number of permutations of length $n$ that avoid the pattern $(1\dd 2\dd3)$ is given by the $nth$ Catalan number, $C_n$. Therefore, the number of possible ways we can arrange the elements in $R$, in such a way that $R$ avoids $(1\dd 2\dd 3)$, is given by $C_{i-1}$. The component $L$ is similar except $i$ cannot be the rightmost element. Without such a restriction, the number of ways we could arrange the elements in $L$ would be given by $C_{n-i+1}$. From this number we must subtract $C_{n-i}$, the number of all permutations of length $n-i+1$ such that the lowest element is fixed in the rightmost position. Hence, the total number of possible permutations for a given $i$ is $(C_{n-i+1}-C_{n-i})C_{i-1}$. To obtain the total number of permutations, we sum over all possible $i$, which cannot be smaller than $2$, and no greater than $n-1$, to ensure that both $L$ and $R$ are nonempty and satisfy the restrictions. The number is therefore given by

\begin{equation}
 \sum_{i=2}^{n-1} (C_{n-i+1}-C_{n-i})C_{i-1}.
\end{equation} 

By using basic algebra and the fact that the Catalan function $C(x)$ satisfies $C(x)=1+xC(x)^2$, we obtain the following when we consider the corresponding generating function for (2).

\begin{eqnarray*}
 &\sum_{n \geq 3}{\sum_{i=2}^{n-1}(C_{n-i+1}-C_{n-i})C_{i-1}x^n} \\
=&x^{-1}(C(x)-xC(x))xC(x)-2C(x)+xC(x)+1 \\
=&x^3C(x)^5.
\end{eqnarray*}

The part of the summand that reads $-2C(x)+xC(x)$ in the initial expression above stems from the range of $i$ being limited. Referring back to the bivariate generating function in Corollary 5, we see that if we put $d=4$ and divide the function with $x$, we get $x^3C(x)^5$ as well. We then replace $n$ with $n+1$ in the associated counting formula to obtain the stated results.\end{proof}

\subsection{Permutations containing a single occurrence of an increasing subsequence of length three}

We are now equipped to handle the case we mentioned earlier, that is, to find the number of permutations that have a single instance of a $(1\dd 2\dd 3)$ pattern, without any additional restrictions. The result was first established in \cite{Noonan}, and later in \cite{Little}. The generating function and counting formula in the following theorem correspond to $d=5$ in Corollary 5, where we divide the generating function with $x^2$, and replace $n$ with $n+2$ in the counting formula to produce a shift of $2$.
 
\begin{theorem}
Let $c_n$ be the number of permutations $\pi \in S_n^1 (1\dd 2\dd 3)$. Then
$$\sum_{n\geq3} c_n x^n \;=\; x^3 C(x)^6 \;=\; \sum_{n\geq 3} \frac{6}{n+3}{2n-1 \choose n+2} x^n.$$
\end{theorem}

\begin{proof}

We will proceed as we did for Theorem 8, that is, find a bijection between our class of permutations and another class that is easier to enumerate.

Let $\pi \in S_n^1 (1\dd 2\dd 3)$. Let the single occurrence of the pattern $(1\dd 2\dd 3)$ in the permutation consist of $\pi_j=a, \pi_k=b$ and $\pi_{m}=c$, with $j\leq k \leq m $. We see that since there is no other occurrence of $(1\dd 2\dd 3)$ in $\pi$, all the elements to the right of $c$ must be smaller than $b$, or we would have more than one occurrence of a $(1\dd 2\dd 3)$ pattern. Also, we see that the only element to the left of $b$ that is smaller than $b$ is $a$ or else we would have more than one occurrence of a $(1\dd 2\dd 3)$ pattern. Our bijection, which we shall refer to as $\tau$, is as follows. Place $a$ in the position of $b$, place $b$ in the original position of $a$, place $c$ immediately to the left of $a$, and insert a new $b$ into the original location of $c$ (so at this point we have two instances of $b$ in the permutation). Finally, increment all the values to the left of $a$ by one. At this point we have to the left of $a$ the elements $b+1,\ldots,n+1$, and to the right of $a$ we have the elements $1,\ldots,a-1,a+1,\ldots,b$. In other words, we now have a permutation decomposable into a left and right component, which we shall refer to hereafter as $L$ and $R$ respectively, with the left component containing the values $b+1,b+2,\ldots,n+1$, and the right component containing the values $1,2,\ldots,b$ (so $a$ belongs to $R$). We see that $b+1$ cannot be the rightmost element in $L$ because the bijection reserves that position for $c+1$. In addition, we see that $b$ cannot be the leftmost element in $R$ since the bijection reserves that position for $a$. In symbols, we would denote this as

$$\ell_1,\ell_2,\ldots,a,\ldots,b,r_1,r_2,\ldots,c,\ldots,r_{b-2},$$ becoming the permutation
$$\ell_1+1,\ell_2+1,\ldots,b+1,\ldots,c,a,r_1,r_2,\ldots,b,\ldots,r_{b-2},$$ where $\ell_k\in L$, $r_k\in R$.

We can see that this is a bijection by showing that the inverse function is well defined. Let $\pi$ be a permutation of length $n+1$ that is decomposable into a left and right component, $L$ and $R$ respectively, such that the elements $1,\ldots,b$ are in $R$ and the elements $b+1,\ldots,n+1$ are in $L$. Furthermore, let us require that both $L$ and $R$ avoid the $(1\dd 2\dd 3)$ pattern, that $b+1$ is not the rightmost element of $L$, and that $b$ is not the leftmost element of $R$. Let us refer to the set of these permutations as $J'_{n+1,b}$. The inverse of the function we described earlier would then be to decrease all the numbers in $L$, and swap the smallest element, $b+1$, in $L$ with the leftmost element, $r_\ell$, in $R$. Lastly, we would replace $b$ in $R$ with the rightmost element, $\ell_r$, in $L$. 

We now proceed to enumerate $J'_{n+1,i}$. Let $\pi \in J'_{n+1,i}$, and let the decomposition be in such a way that the elements $1,2,\ldots,i$ are in $R$ and $i+1,i+2,\ldots,n+1$ are in $L$. In the proof of Theorem 8 we used the fact that the number of permutations of length $n$ that avoid the pattern $(1\dd 2\dd 3)$ is given by the $nth$ Catalan number, $C_n$. So, the total number of ways we can permute the elements in $L$ is given by $C_{n-i+1}-C_{n-i}$, where we  subtract the number of ways we can arrange the elements in $L$ with $i+1$ fixed in the rightmost position in $L$, since we have excluded all such permutations from $J'_{n+1,i}$. Likewise, we see that the total number of ways we can permute the elements in $R$ is given by $C_{i}-C_{i-1}$, where we subtract the number of ways we can arrange the elements in $R$ with $i$ being in the leftmost position, since we have also excluded those permutations from $J'_{n+1,i}$. Since $\pi$ is simply the concatenation $LR$, we see that the total number of elements in $J'_{n+1,i}$ is given by $(C_{n-i+1}-C_{n-i})(C_i-C_{i-1})$. We then sum over $i \in \left\{2,\ldots,n-1\right\}$;

\begin{equation}
 \sum_{i=2}^{n-1} (C_{n-i+1}-C_{n-i})(C_i-C_{i-1}).
\end{equation}

We recall that the Catalan function $C(x)$ satisfies $C(x)=1+xC(x)^2$, and so when we translate (3) into a generating function we get the following.

\begin{eqnarray*}
 &\sum_{n \geq 3}{\sum_{i=2}^{n-1}{(C_{n-i+1}-C_{n-i})(C_i-C_{i-1})}x^n} \\
=&x^{-1}(C(x)-xC(x))(C(x)-xC(x))-2(C(x)-xC(x)) \\
=&x^3C(x)^6. \\
\end{eqnarray*}

We can then extract the $nth$ coefficient by standard means, to produce the counting formula in Theorem 9.\end{proof}

\subsection{Permutations with a single occurrence of an increasing subsequence of length three, with two elements nonadjacent.}

The generating function and counting formula in the following theorem correspond to $d=6$ in Corollary 5, where we divide the generating function by $x^2$, and replace $n$ with $n+2$ to produce a shift of $2$.

\begin{theorem}
Let $d_n$ be the number of permutations $\pi \in S_n^1 (1\dd 2\dd 3)$ for which $(1\dd 23)\pi=0$. Then
$$\sum_{n\geq4} d_n x^n \;=\; x^4 C(x)^7 \;=\; \sum_{n\geq 4} \frac{7}{n+3}{2n-2 \choose n+2}x^n.$$

\end{theorem}

\begin{proof}

Our proof will be nearly identical to the one we constructed for Theorem 9, with the exception that an extra restriction is required for $L$, to eliminate adjacency. 

Let $\pi \in S_n^1(1\dd 2\dd 3)$ with $(1\dd 23)\pi = 0$. Let the single occurrence of the pattern in the permutation consist of $\pi_j=a, \pi_k=b$ and $\pi_{m}=c$, with $j\leq k \leq m-1 $, we see that since there is no other occurrence of $(1\dd 2\dd 3)$ in $\pi$, all the elements to the right of $c$ must be less than $b$. Also, we see that the only element to the left of $b$ that is smaller than $b$ is $a$. We will now apply the bijection $\tau$ we introduced in the proof of Theorem 9. In short, we place $a$ in the position of $b$, place $b$ in the original position of $a$, place $c$ immediately to the left of $a$, and insert a new $b$ into the original location of $c$. Lastly, we increment all the values to the left of $a$ by one. Our permutation now has the familiar structure, $LK$, with $L$ consisting of the elements $b+1,b+2,\ldots,n+1$, and $R$ consisting of the elements $1,2,\ldots,b$. In symbols, we have the permutation

$$\ell_1,\ell_2,\ldots,a,\ldots,b,r_1,r_2,\ldots,c,\ldots,r_{b-2},$$ transformed into the permutation
$$\ell_1+1,\ell_2+1,\ldots,b+1,\ldots,c,a,r_1,r_2,\ldots,b,\ldots,r_{b-2},$$

Let us denote the set of all the latter permutations as $J''_{n+1,b}$. There are several observations to be made here. As before, $b+1$ cannot be the rightmost element in $L$ because the bijection reserves that position for $c+1$, and $b$ cannot be the leftmost element in $R$ since the bijection reserves that position for $a$. In addition, $b$ cannot be the second leftmost element in $R$ because none of the permutations in the domain of the bijection here have $b$ and $c$ adjacent, so once $a$ is added to $R$, there are at least two elements in front of the new $b$ in $R$. 

We now determine the cardinality of $J''_{n+1,i}$. In the proof of Theorem 8, we demonstrated that the number of ways in which we can permute the elements of $L$ is given by $C_{n-i+1}-C_{n-i}$. Without the extra restriction we have placed on $R$ in this proof, the number of ways we can permute the elements in $R$ is likewise shown to be $C_{i}-C_{i-1}$. Since we have disallowed the possibility of having $i$ as the second leftmost element in $R$, we note that $i$ is the largest element in the component and as such cannot possibly contribute to any instances of a $(1\dd 2\dd 3)$ pattern if it remains in the leftmost, or second leftmost, position, and so the number of permutations of length $i$ that avoid the pattern $(1\dd 2\dd 3)$, with $i$ being neither the leftmost nor the second leftmost element, is given by $C_{i}-2C_{i-1}$, where we have subtracted $C_{i-1}$ twice, once for each forbidden location of $i$. The number of elements in $J''_{n+1,i}$ is thus given by $(C_{n-i+1}-C_{n-i})(C_i-2C_{i-1})$. We now sum over $i \in \left\{3,\ldots,n-1\right\}$ and find the number to be

\begin{equation}
 \sum_{i=3}^{n-1} (C_{n-i}-C_{n-i-1})(C_i-2C_{i-1}).
\end{equation}

We remember that the Catalan function $C(x)$ satisfies $C(x)=1+xC(x)^2$, and thus obtain the following when we translate (4) into a generating function.

\begin{eqnarray*}
 &\sum_{n \geq 4}{\sum_{i=3}^{n-1} (C_{n-i}-C_{n-i-1})(C_i-2C_{i-1})x^n} \\
=&x^{-1}(C(x)-xC(x))(C(x)-2xC(x))+(C(x)-xC(x))\\
 &-(x^{-1}C(x)-C(x))-(x^{-1}C(x)-2C(x))-1+x^{-1} \\
=&x^4C(x)^7. \\
\end{eqnarray*}

We can then extract the $nth$ coefficient by standard means, to produce the counting formula in Theorem 10.\end{proof}

\subsection{Permutations with a single increasing subsequence of length three, with no two elements adjacent.}

The generating function and counting formula in the theorem below correspond to $d=7$ in Corollary 5, where we divide the generating function by $x^2$, and replace $n$ with $n+2$ to produce a shift of $2$.

\begin{theorem}
Let $f_n$ be the number of permutations $\pi \in S_n^1 (1\dd 2\dd 3)$ for which $(1\dd 23)\pi=0$ and $(12\dd 3)\pi=0$. Then
$$\sum_{n\geq5} f_n x^n \;=\; x^5 C(x)^8 \;=\; \sum_{n\geq 5} \frac{8}{n+3}{2n-3 \choose n+2}x^n.$$
\end{theorem}

\begin{proof}

Our proof will be nearly identical to the one we constructed for Theorem 10, with the exception that an extra restriction is required for $R$, to eliminate adjacency. 

Let $\pi \in S_n^1(1\dd 2\dd 3)$ with $(1\dd 23)\pi = 0$ and $(12\dd 3)\pi = 0$. Let the single occurrence of the pattern in the permutation consist of $\pi_j=a, \pi_k=b$ and $\pi_{m}=c$, with $j\leq k-1 \leq m-2 $. We see that since there is no other occurrence of $(1\dd 2\dd 3)$ in $\pi$, all the elements to the right of $c$ must be smaller than $b$. Also, we see that the only element to the left of $b$ that is smaller than $b$ is $a$. We will once again apply the bijection $\tau$ we introduced in the proof of Theorem 9. That is, we place $a$ in the position of $b$, place $b$ in the original position of $a$, place $c$ immediately to the left of $a$, and insert a new $b$ into the original location of $c$. Lastly, we increment all the values to the left of $a$ by one. Our permutation now has the familiar structure, $LK$, with $L$ consisting of the elements $b+1,b+2,\ldots,n+1$, and $R$ consisting of the elements $1,2,\ldots,b$. In symbols, we have the permutation

$$\ell_1,\ell_2,\ldots,a,\ldots,b,r_1,r_2,\ldots,c,\ldots,r_{b-2},$$ transformed into the permutation
$$\ell_1+1,\ell_2+1,\ldots,b+1,\ldots,c,a,r_1,r_2,\ldots,b,\ldots,r_{b-2},$$

Let us denote the set containing all the permutations of the latter kind by $J'''_{n+1,b}$. The observations we made on $J''_{n+1,b}$ apply here as well. We see that $b+1$ cannot be the rightmost element in $L$, $b$ cannot be the leftmost element in $R$, and $b$ cannot be the second leftmost element in $R$ either. Since none of the permutations in the domain of the bijection have $c$ as the second rightmost element in $L$, we also see that $b$ cannot be the second leftmost element in $L$.

We now enumerate $J'''_{n+1,i}$. In the proof of Theorem 10, we demonstrated that the number of ways in which we can permute the elements of $R$ is given by $C_{i}-2C_{i-1}$. We now have exactly the same situation with $L$, since $i+1$ can neither be the rightmost, nor the second rightmost element in that component. Since $i+1$ is the smallest element however, we know that placing it in either of these locations cannot contribute to any occurrences of a $(1\dd 2\dd 3)$ pattern in $L$, and so it becomes straightforward to account for the permutations where this happens. We simply subtract $2C_{n-i}$ to find the total number of possible ways of permuting the elements in $L$ to be $C_{n-i+1}-2C_{n-i}$. Hence, the number of elements in $J'''_{n+1,i}$ is given by $(C_{n-i+1}-2C_{n-i})(C_{i}-2C_{i-1})$. We sum over $i\in \left\{3,\ldots,n-2\right\}$;

\begin{equation}
 \sum_{i=3}^{n-2} (C_{n-i}-2C_{n-i-1})(C_i-2C_{i-1}).
\end{equation}

We remind ourselves that the Catalan function $C(x)$ satisfies $C(x)=1+xC(x)$, and thus we obtain the following when we translate (5) into a generating function.

\begin{eqnarray*}
 &\sum_{n \geq 5}{\sum_{i=3}^{n-2} {(C_{n-i}-2C_{n-i-1})(C_i-2C_{i-1})}x^n} \\
=&x^{-1}(C(x)-2xC(x))^2+2(C(x)-2xC(x))-2(x^{-1}C(x)-2C(x))+x-2+x^{-1}\\
=&x^5C(x)^8. \\
\end{eqnarray*}

We can then extract the $nth$ coefficient by standard means, to produce the counting formula in Theorem 11.\end{proof}

\section{Further remarks}

\subsection{A Game theoretic application}

As it turns out, the counting formula in Theorem 4 lends itself to considerations of a basic card game \cite{MC}. Suppose we have a deck of $n$ red cards and $n$ black cards and we play the following game:  We draw cards, one by one, from the top of the deck, and we can stop at any point.  When we stop, we get a score that is the number of red cards minus the number of black cards drawn, or zero if this number is negative.  If our strategy is to stop once we have reached a score of $r$, what should $r$ be in order to maximize our expected score, given a deck of $2n$ cards? To answer that question we need to know the probability that the cards in the deck are arranged in such a way that we will at some point have drawn $r$ more red cards than black cards. Let $P_{r,2n}$ stand for the number of ways we can arrange a deck of $2n$ cards in such a way that if we draw all the cards from it the highest score we will come across is $r$. In other words, $P_{r,2n}$ stands for the number of permutations of the cards in the deck such that drawing from it yields at some point an excess of $r$ red cards over black cards, but no more than that. 

Now, we are looking for the number of ways we can arrange the cards in the deck so that we are able to draw 'at least' $r$ more red cards than black cards, and so the number of ways we can arrange the cards for that to happen is given by $P_{r,2n}+P_{r+1,2n}+P_{r+2,2n}+\ldots+P_{n,2n} = \sum_{i=r}^n {P_{i,2n}}$. The probability of this happening is then given by $$\frac{1}{{2n \choose n}}\sum_{i=r}^n P_{i,2n}.$$

We now want to find $P_{r,2n}$. Every permutation of a deck with $2n$ cards can be considered to be a lattice path. Going through the deck by drawing one card at a time, we translate each red card to an upstep $(1,1)$, and each black card to a downstep $(1,-1)$. If we start at the origin, the path should end at $(2n,0)$, since the number of red cards and black cards are equal. If the highest score we are able to get by drawing from the deck is $r$ then that means the lattice path corresponding to the arrangement of the cards in the deck reaches a maximum height of $r$ over the $x$-axis. In the following theorem, we give a bijection between the set of lattice paths of length $2n-2r$ that start at the origin, reach a height $r$ over the $x$-axis but do not exceed that height, and end in $(2n-2r,0)$, and the set of Dyck paths of semilength $n$ that start with $r$ upsteps, end with $r$ downsteps and touch the $x$-axis somewhere between the endpoints. The reason for why the lengths of the corresponding objects from each family differ by $2r$ is that in the family of Dyck paths we have just described, the $r$ upsteps in the beginning, and $r$ downsteps in the end are fixed, and so do not contribute to the number of ways we can create such a path.

\begin{theorem}
The number of lattice paths that begin at the origin, $(0,0)$, consist only of upsteps $(1,1)$ and downsteps $(1,-1)$, end at $(2n-2r,0)$, and touch the line $y=r$ at least once but do not cross it, has generating function $x^{r}C(x)^{2r+1}$, and counting formula	
\[
	\frac{2r+1}{n+r+1}{2n \choose n+r}.
\]
\end{theorem}

\begin{proof}
Let $\lambda$ be a lattice path consisting only of upsteps and downsteps, $(1,1)$ and $(1,-1)$ respectively, that begins at the origin and ends at $(2n-2r,0)$. The following set of operations defines a bijection from the set of all such $\lambda$ to the set of Dyck paths of semilength $n$ that start with $r$ upsteps, end with $r$ downsteps and touch the $x$-axis somewhere between the endpoints. Shift $\lambda$ by $r$ units to the right in the plane, and $r$ units down, so that it now starts in $(r,-r)$. Now reflect the resulting path about the $x$-axis. We now have a path that starts in $(r,r)$, and ends in $(2n-r,r)$, and touches the $x$-axis somewhere between the end points. Lastly, append $r$ consecutive upsteps to the beginning of the path, and $r$ consecutive downsteps to the end of it. The path now starts in $(0,0)$ and ends in $(2n,0)$. This process is demonstrated in Figure \ref{fig4}. The enumeration of this type of paths has already been dealt with in Theorem 4. Since $p$ and $k$ are both equal to $r$ here, the generating function becomes $x^{2r}C^{2r+1}$. However, we have added $2r$ steps to $\lambda$ so we need to divide the generating function by $x^r$. The result is the function $x^{r}C^{2r+1}$. The $nth$ coefficient of this generating function can then be extracted by standard methods.\end{proof}

\begin{figure}
\newcommand\p{\circle*{1}}

\setlength{\unitlength}{1mm}
\begin{picture}(140,100)(33,-10)

\allinethickness{.03mm}

\put(50,60){
\put(0,0){\grid(90,20)(5,5)}
\put(0,-10){
\path(0,20)(5,15)(10,10)(15,15)(20,10)(25,15)(30,20)(35,25)(40,30)(45,25)(50,30)(55,25)(60,30)(65,25)(70,20)
}}

\put(-55,0){
\put(105,0){\grid(90,20)(5,5)}
\path(105,0)(110,5)(115,10)(120,15)(125,20)(130,15)(135,20)(140,15)(145,10)(150,5)(155,0)(160,5)(165,0)(170,5)(175,0)(180,5)(185,10)(190,5)(195,0)
}

\put(-55,30){
\put(105,0){\grid(90,20)(5,5)}
\path(115,10)(120,15)(125,20)(130,15)(135,20)(140,15)(145,10)(150,5)(155,0)(160,5)(165,0)(170,5)(175,0)(180,5)(185,10)}

\put(95,55){$\updownarrow$}
\put(95,25){$\updownarrow$}

\allinethickness{.5mm}
\path(50,0)(140,0)
\path(50,30)(140,30)
\path(50,70)(140,70)
\end{picture}
\caption{A lattice path representing a particular arrangement of the cards (top), the path after we have flipped it over and translated it $r$ units to the right and $r$ units up (middle), and the corresponding Dyck path under the bijection in Theorem 12 (bottom).}
\label{fig4}
\end{figure}
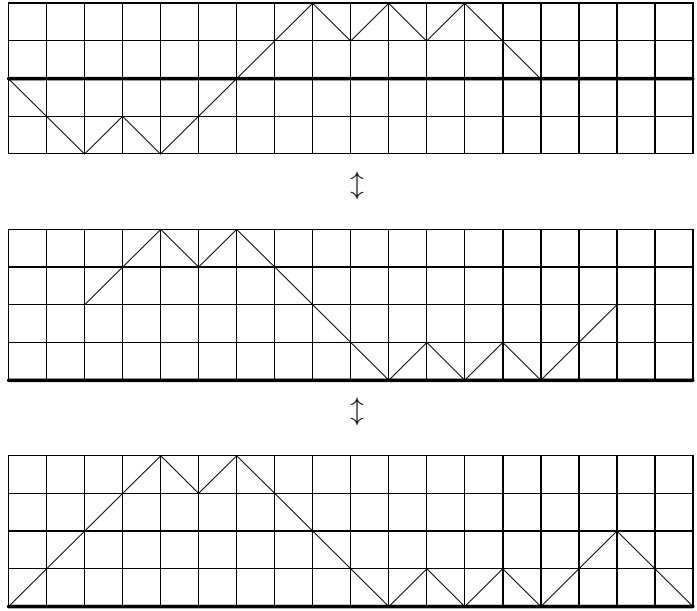

So we see that $$P_{r,2n}=\frac{2r+1}{n+r+1}{2n \choose n+r}.$$ Bearing in mind our considerations in the first paragraph of this section, we then see that the probability of being able to achieve a score $r$ by drawing from a deck of $2n$ cards, is given by $$\frac{1}{{2n\choose n}}\sum_{k=r}^n  {\frac{2k+1}{n+k+1}{{2n \choose n+k}}}.$$ Thus, if our strategy is to stop once we have reached a score of $r$, the expected value is $$E_{r}(n)=\frac{r}{{2n \choose n}}\sum_{k=r}^n  {\frac{2k+1}{n+k+1}{{2n \choose n+k}}},$$ since each successful run will yield a score of $r$. Since we are looking for the value of $r$ that maximizes $E_{r}(n)$ for a given $n$ we must then look at the set of expected values we get for all possible values of $r$. 

Now, we might be interested in trying to find a function that maps $n$ to the value of $r$ that gives the highest expected return for a deck of cards of size $2n$. If we write out the optimal value of $r$ for the first few $n$, we obtain the following sequence, $$1,1,1,1,2,2,2,2,2,2,2,2,3,3,3,3,3,3,3,3,3,3,3,3,4,\ldots$$ When we inspect the sequence, we see that $1$ appears four times, $2$ appears eight times, $3$ appears $12$ times, $4$ appears $16$ times, and so on. Based on this observation we conjecture that the value of $r$ that gives the highest expected return for a given $n$ is given implicitely by the equation $n=max_r\left\{{4{r+1 \choose 2}} \leq n\right\}$.

\end{document}